\documentclass[11pt]{article}
\usepackage{tikz,pgfplots,latexsym,amssymb,amsmath,amsthm,enumerate,geometry,float,cite}
\pgfplotsset{width=6cm,compat=newest}
\newtheorem{theorem}{Theorem}
\newtheorem{lemma}[theorem]{Lemma}

\usepackage{lineno}
\usepackage{setspace}

\begin{document}
%\linenumbers
\onehalfspace

\title{Vertex degrees close to the average degree}
\author{
Johannes Pardey\and Dieter Rautenbach}
\date{}

\maketitle
\vspace{-10mm}
\begin{center}
{\small 
Institute of Optimization and Operations Research, Ulm University, Ulm, Germany\\
\texttt{$\{$johannes.pardey,dieter.rautenbach$\}$@uni-ulm.de}
}
\end{center}

\begin{abstract}
Let $G$ be a finite, simple, and undirected graph 
of order $n$ and average degree $d$.
Up to terms of smaller order, we characterize the minimal intervals $I$ 
containing $d$ that are guaranteed to contain some vertex degree.
In particular, for $d_+\in \left(\sqrt{dn},n-1\right]$, 
we show the existence of a vertex in $G$ of degree between
$d_+-\left(\frac{(d_+-d)n}{n-d_++\sqrt{d_+^2-dn}}\right)$ and $d_+$.
\\[3mm]
\end{abstract}

\section{Introduction}

An obvious observation concerning finite, simple, and undirected graphs is
that every such graph $G$ with $n$ vertices and average degree $d$ 
has a vertex of degree at most $d$
as well as a vertex of degree at least $d$,
that is, there are vertex degrees in the two intervals $[0,d]$ and $[d,n-1]$.
In the present note we study minimal intervals $I$ containing $d$
for which every such graph $G$ 
necessarily contains some vertex $u$ whose degree $d_G(u)$ lies in $I$.
Surely, which intervals have this property is implicit 
in characterizations of degree sequences 
such as the well-known Erd\H{o}s-Gallai characterization \cite{erga,hasc}.
Our goal here is to specify such intervals explicitly,
and our motivation originally came from embedding problems
that rely on special cases of the results presented here \cite{mopara,para}.

Our first result specifies a natural inverval of length about half the order 
around the average degree that is guaranteed to contain some vertex degree.
Some of the estimates in its proof, 
cf.~(\ref{e6}) and (\ref{e7}) below,
correspond to inequalities from the Erd\H{o}s-Gallai characterization \cite{erga}.

\begin{theorem}\label{theorem1}
If $G$ is a graph with $n$ vertices and $m$ edges, 
then there is a vertex $u$ in $G$ with
\begin{align}\label{e1}
d-\frac{n-2}{2(n-1)}d\leq d_G(u)\leq d+\frac{n-2}{2(n-1)}\overline{d},
\end{align}
where $d=\frac{2m}{n}$ is the average degree of $G$ 
and $\overline{d}=n-1-d$ is the average degree of 
the complement $\overline{G}$ of $G$.

Furthermore, if $G$ has no vertex $u$ with 
$d-\frac{n-2}{2(n-1)}d<d_G(u)<d+\frac{n-2}{2(n-1)}\overline{d}$, then 
\begin{itemize}
\item the vertex set of $G$ is the disjoint union of two sets $V_+$ and $V_-$,
\item $V_+$ is a clique of order $\frac{dn}{n-1}$,
\item $V_-$ is an independent set of order $\frac{\overline{d}n}{n-1}$,
\item every vertex in $V_+$ is adjacent to exactly half the vertices in $V_-$, and
\item every vertex in $V_-$ is adjacent to exactly half the vertices in $V_+$.
\end{itemize}
\end{theorem}
Theorem \ref{theorem1} gives a rather precise answer in a specific setting,
where the interval depends on $n$ and $d$;
allowing for the characterization of the corresponding extremal graphs.
Our second result applies in a more general setting;
it gives the precise answer up to terms of smaller order.
 
\begin{theorem}\label{theorem2}
If $G$ is a graph of order $n$ and average degree $d$ with $0<d<n-1$,
and $d_+\in \left(\sqrt{dn},n-1\right]$, 
then there is a vertex $u$ in $G$ with
\begin{align*}
d_+-\frac{(d_+-d)n}{n-d_++\sqrt{d_+^2-dn}}
\leq d_G(u)\leq d_+.
\end{align*}
\end{theorem}
Let $d_-$ denote the lower bound for $d_G(u)$ specified in Theorem \ref{theorem2}.
For $d_+\leq \sqrt{dn}$, 
Lemma \ref{lemma1} below indicates 
that nothing really non-trivial can be said about $d_-$,
more precisely, in this case, $d_-$ will be $0$.
Theorem \ref{theorem2} gives an estimate for $d_-$ 
that is best possible up to terms of smaller order;
one can construct suitable almost extremal graphs 
approximating the values in (\ref{eopt}) below.
Applying Theorem \ref{theorem2} to the complement of $G$
yields a symmetric result, where one first specifies the lower bound 
$d_-$ for $d_G(u)$ and then determines the upper bound $d_+$ accordingly,
that is, in this case, $d_+$ is considered as a function of $n$, $d$, and $d_-$.
For given $n$ and $d$ as in Theorem \ref{theorem2}, 
the length
$$\ell_{\min}(d_+)
=d_+-d_-
=\frac{(d_+-d)n}{n-d_++\sqrt{d_+^2-dn}}
=\left(\frac{\frac{d_+}{n}-\frac{d}{n}}{1-\frac{d_+}{n}+\sqrt{\left(\frac{d_+}{n}\right)^2-\frac{d}{n}}}\right)n$$
of the specified interval satisfies
$\min\left\{\ell_{\min}(d_+):
d_+\in \left(\sqrt{dn},n-1\right]\right\}\leq \frac{n}{2}$
for every $d\in (0,n-1)$.
Up to terms of smaller order, 
Theorem \ref{theorem1} corresponds to the choice
$\frac{d_+}{n}=\frac{n+d}{2n}=\frac{1}{2}+\frac{d}{2n}$,
in which case
$\ell_{\min}(d_+)=\frac{n}{2}$,
that is, up to terms of smaller order, Theorem \ref{theorem2}
implies Theorem \ref{theorem1}.
For $d\not=\frac{n}{2}$ and suitable choices of $d_+$, 
Theorem \ref{theorem2} 
guarantees a vertex degree within a smaller interval around $d$ 
than Theorem \ref{theorem1},
cf.~Figure \ref{plot}.

\begin{figure}[h]
{\footnotesize
\begin{tikzpicture}
\begin{axis}[xlabel=$\frac{d_+}{n}$,
%ylabel=$\frac{\ell_{\min}(n,0.25n,d_+)}{n}$,
domain = 0.5:1,domain y = 0:1,samples=200] 
 \addplot[mark=](x,{(x-0.25)/(1-x+sqrt(x^2-0.25))});
\end{axis}
\end{tikzpicture}
\begin{tikzpicture}
\begin{axis}[xlabel=$\frac{d_+}{n}$,
%ylabel=$\frac{\ell_{\min}(n,0.5n,d_+)}{n}$,
domain = 0.7072:1,domain y = 0:1,samples=100] 
 \addplot[mark=](x,{(x-0.5)/(1-x+sqrt(x^2-0.5))});
\end{axis}
\end{tikzpicture}
\begin{tikzpicture}
\begin{axis}[xlabel=$\frac{d_+}{n}$,
%ylabel=$\frac{\ell_{\min}(n,0.81n,d_+)}{n}$,
domain = 0.9:1,domain y = 0:1,samples=100] 
 \addplot[mark=](x,{(x-0.81)/(1-x+sqrt(x^2-0.81))});
\end{axis}
\end{tikzpicture}
}
\caption{The three plots show $\frac{\ell_{\min}(d_+)}{n}$
as a function of $\frac{d_+}{n}\in\left(\sqrt{\frac{d}{n}},1\right]$ for $\frac{d}{n}\in \{ 0.25,0.5,0.81\}$.}\label{plot}
\end{figure}
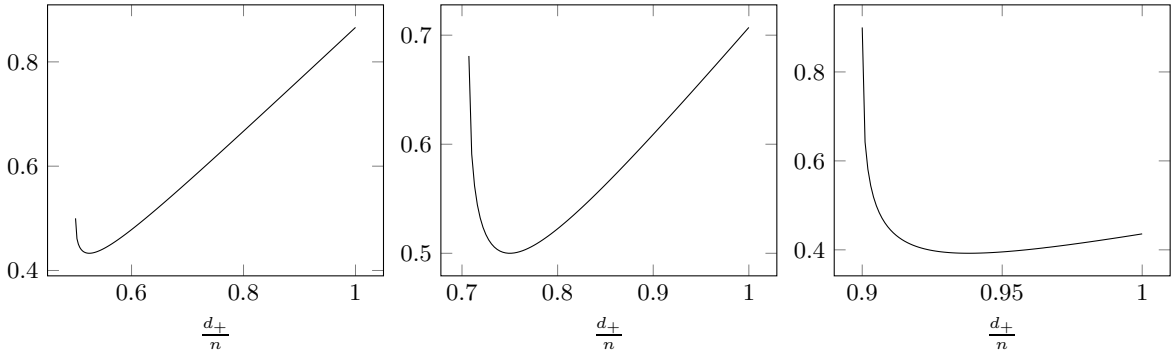
Applying our results repeatedly to a given graph,
each times removing a vertex of degree close to the current average degree,
yields several vertices whose degrees are guaranteed to lie in 
slowly changing intervals around the original average degree.

The proofs are given in the following section.

\section{Proofs}

We proceed to the proofs of our two results.

\begin{proof}[Proof of Theorem \ref{theorem1}.]
Let $G$, $n$, $m$, $d$, and $\overline{d}$ be as in the statement.
Since the statement is trivial for $d\in \{ 0,n-1\}$, we may assume that $0<d<n-1$.
Let 
\begin{align*}
V_-=\left\{ u\in V(G):d_G(u)\leq d-\frac{n-2}{2(n-1)}d\right\}
&&\mbox{ and }&&
V_+&=\left\{ u\in V(G):d_G(u)\geq d+\frac{n-2}{2(n-1)}\overline{d}\right\}.
\end{align*}
In view of the desired statements, 
we may assume that $V(G)$ is the disjoint union of $V_+$ and $V_-$.
Since $G$ has average degree $d$, 
the sets $V_+$ and $V_-$ are both non-empty, and, hence, 
we have $\frac{1}{n}\leq x\leq 1-\frac{1}{n}$ for $x=\frac{|V_+|}{n}$.
The degree sum $\sum\limits_{u\in V_+}d_G(u)$ of the vertices in $V_+$ 
is at most the sum of the following two terms:
\begin{itemize}
\item Twice the number of edges in $G$ between vertices in $V_+$.
This is at most $xn(xn-1)$
with equality if and only if $V_+$ is a clique.
\item The number of edges in $G$ between $V_+$ and $V_-$.
This is at most $\sum\limits_{u\in V_-}d_G(u)$
with equality if and only if $V_-$ is an independent set.
\end{itemize}
This implies
\begin{align}
dn &=\sum_{u\in V_+}d_G(u)+\sum_{u\in V_-}d_G(u)\nonumber\\
& \leq xn(xn-1)+2\sum_{u\in V_-}d_G(u)\label{e2}\\
&\leq  xn(xn-1)+2\underbrace{(1-x)n}_{=|V_-|}
\underbrace{\left(d-\frac{n-2}{2(n-1)}d\right)}_{\geq d_G(u)\,\,for\,\,u\in V_-}.\label{e3}
\end{align}
Applying exactly the same couting argument to $\overline{G}$ implies
\begin{align}
(n-1-d)n &=\sum_{u\in V_-}d_{\overline{G}}(u)+\sum_{u\in V_+}d_{\overline{G}}(u)\nonumber\\
&\leq (1-x)n((1-x)n-1)+2\sum_{u\in V_+}(n-1-d_G(u))\label{e4}\\
&\leq (1-x)n((1-x)n-1)+2xn
\underbrace{\left(n-1-d-\frac{n-2}{2(n-1)}\overline{d}\right)}_{\geq d_{\overline{G}}(u)\,\,for\,\,u\in V_+}.\label{e5}
\end{align}
By (\ref{e3}) and (\ref{e5}), we have
\begin{align}
f_1(x)&:=x(xn-1)+2(1-x)\left(d-\frac{n-2}{2(n-1)}d\right)-d
=\frac{(xn-1)(x(n-1)-d)}{n-1}\geq 0\mbox{ and }\label{e6}\\
f_2(x)&:=(1-x)((1-x)n-1)+2x\left(n-1-d-\frac{n-2}{2(n-1)}\overline{d}\right)-(n-1-d)\nonumber\\
&=\frac{(x(n-1)-d)((x-1)n+1)}{n-1}\geq 0.\label{e7}
\end{align}
Recall that $x\in \left[\frac{1}{n},1-\frac{1}{n}\right]$.
Since 
$f_1\left(\frac{1}{n}\right)=0$,
$f_2\left(1-\frac{1}{n}\right)=0$, 
$f_1\left(\frac{d}{n-1}\right)=f_2\left(\frac{d}{n-1}\right)=0$, and
$f_1(x)$ as well as $f_2(x)$ are both strictly convex, we have
$\min\{ f_1(x),f_2(x)\}<0$ for $x\in \left[\frac{1}{n},1-\frac{1}{n}\right]\setminus \left\{ \frac{1}{n},\frac{d}{n-1},1-\frac{1}{n}\right\}$.
By (\ref{e6}) and (\ref{e7}), 
this implies
\begin{align*}
x&\in \left\{ \frac{1}{n},\frac{d}{n-1},1-\frac{1}{n}\right\}
&\mbox{ and }&&\min\Big\{ f_1(x),f_2(x)\Big\}=0.
\end{align*}
In order to show the existence of a vertex $u$ in $G$ that satisfies (\ref{e1}), 
we suppose, for a contradiction, 
that no vertex in $V_+$ has degree $d+\frac{n-2}{2(n-1)}\overline{d}$ and
that no vertex in $V_-$ has degree $d-\frac{n-2}{2(n-1)}d$.
Since $V_+$ and $V_-$ are both non-empty, 
the two inequalities (\ref{e3}) and (\ref{e5}) become strict,
which implies the contradiction $\min\{ f_1(x),f_2(x)\}>0$.
Hence, some vertex $u$ in $G$ satisfies (\ref{e1}).

We proceed to the proof of the second part of the statement.
Note that the hypothesis that 
$G$ has no vertex $u$ with $d-\frac{n-2}{2(n-1)}d<d_G(u)<d+\frac{n-2}{2(n-1)}\overline{d}$
is equivalent to our assumption that 
$V(G)$ is the disjoint union of $V_+$ and $V_-$.
If $V_+$ is not a clique or $V_-$ is not an independent set, 
then (\ref{e2}) and (\ref{e4}) become strict inequalities, 
and we obtain the contradiction $\min\{ f_1(x),f_2(x)\}>0$.
Similarly, 
if some vertex in $V_+$ has a degree strictly larger than $d+\frac{n-2}{2(n-1)}\overline{d}$
and 
some vertex in $V_-$ has a degree strictly smaller than $d-\frac{n-2}{2(n-1)}d$,
then (\ref{e3}) and (\ref{e5}) become strict inequalities,
and again we obtain the contradiction $\min\{ f_1(x),f_2(x)\}>0$.
If $V_+$ is a clique,
$V_-$ is an independent set, 
all vertices in $V_-$ have degree exactly $d-\frac{n-2}{2(n-1)}d$, and
some vertex in $V_+$ has degree strictly larger than $d+\frac{n-2}{2(n-1)}\overline{d}$,
then 
\begin{align*}
dn & =\sum_{u\in V_+}d_G(u)+\sum_{u\in V_-}d_G(u)\\
& >\left(d+\frac{n-2}{2(n-1)}\overline{d}\right)xn+\left(d-\frac{n-2}{2(n-1)}d\right)(1-x)n\\
& = \frac{(n-2)n}{2}x+\frac{n^2d}{2(n-1)}
\end{align*}
implies the contradiction $x<\frac{d}{n-1}$.
Hence, by symmetry between $G$ and $\overline{G}$,
we may assume that 
$V_+$ is a clique,
$V_-$ is an independent set, 
all vertices in $V_-$ have degree exactly $d-\frac{n-2}{2(n-1)}d$, and
all vertices in $V_+$ have degree exactly $d+\frac{n-2}{2(n-1)}\overline{d}$.
As above, this implies $x=\frac{d}{n-1}$, and it follows easily
that every vertex in $V_+$ is adjacent to exactly half the vertices in $V_-$ 
and 
that every vertex in $V_-$ is adjacent to exactly half the vertices in $V_+$.
This completes the proof.
\end{proof}
The interval 
$$\left[d-\frac{n-2}{2(n-1)}d,d+\frac{n-2}{2(n-1)}\overline{d}\right]$$
from Theorem \ref{theorem1} 
always has the same length $\frac{n-2}{2}$
but it naturally shifts with changes of $d$.
Now, for a given value of $d_+$ from $[d,n-1]$,
we consider the largest $d_-$ from $[0,d]$ 
such that the interval $I=[d_-,d_+]$ around $d$ 
necessarily contains some vertex degree,
that is, we specify the upper end $d_+$ of $I$
and consider the lower end $d_-$ of $I$ as a function of $n$, $d$, and $d_+$.
Quite naturally, if $d_+$ is close enough to $d$, then $d_-$ will have to be $0$,
which leads to the obvious observation at the very beginning of this note.

Now, let $n$ and $m$ be integers with $0<m<{n\choose 2}$,
and let $d=\frac{2m}{n}$.

For $d_+\in (d,n-1)$, let 
$$d_-=d_-(n,d,d_+)$$ 
be the smallest possible value of $d_-\in [0,d]$ such that 
there is some graph $G$ of order $n$ and average degree $d$
whose vertex set $V(G)$ is partitioned into the two sets
\begin{align*}
V_-=\left\{ u\in V(G):d_G(u)\leq d_-\right\}
&&\mbox{ and }&&
V_+&=\left\{ u\in V(G):d_G(u)\geq d_+\right\}.
\end{align*}
Clearly, the choices of $d_-$ and $G$ imply that some vertex in $G$
has degree $d_-$, in particular, $d_-$ is an integer.

This definition implies that 
\begin{itemize}
\item every graph of order $n$ and average degree $d$
has a vertex of degree in $[d_-,d_+]$, 
while 
\item some such graph has no vertex of degree in $(d_-,d_+)$,
\end{itemize}
that is, $[d_-,d_+]$ is indeed a minimal interval with the desired property.
Our goal is to estimate $d_-=d_-(n,d,d_+)$.

Let $G$ be as above.

Similarly as in the proof of Theorem \ref{theorem1},
the condition $d_+>d$ implies that $V_-$ and $V_+$ are both non-empty.

For the average degrees
$\overline{d}_{\pm}=\frac{1}{|V_{\pm}|}\sum\limits_{u\in V_{\pm}}d_G(u)$
within $V_-$ and $V_+$, respectively,
we obtain
\begin{align}\label{ep1}
0\leq \overline{d}_-\leq d_-\leq d<d_+\leq \overline{d}_+\leq n-1.
\end{align}
For $x=\frac{|V_+|}{n}\in (0,1)$, we have
\begin{align}\label{ep2}
(1-x)n\overline{d}_-+xn\overline{d}_+=dn.
\end{align}
Since the number of edges between $V_-$ and $V_+$ is at least
$\left(\overline{d}_+-xn+1\right)xn$, we have the following
Erd\H{o}s-Gallai type inequality
\begin{align}\label{ep3}
(1-x)n\overline{d}_-\geq \left(\overline{d}_+-xn+1\right)xn.
\end{align}
It follows that 
\begin{align}\label{ep9}
d_-(n,d,d_+)\geq {\rm OPT}(P)
\end{align}
for the following optimization problem $(P)$:
\begin{align}
&& \min\,\, & d_-\nonumber\\
&& s.th.\,\, & (1-x)\overline{d}_-+x\overline{d}_+ =d\label{ep4}\\
(P)&&& (1-x)\overline{d}_-\geq \left(\overline{d}_+-xn\right)x\label{ep5}\\
&&& 0\leq \overline{d}_-\leq d_-\label{ep6}\\
&&& d_+\leq \overline{d}_+\leq n\label{ep7}\\
&&& 0\leq x\leq 1\label{ep8}\\
&&& d_-,\overline{d}_-,\overline{d}_+,x\in \mathbb{R}.\nonumber
\end{align}
Note that $(P)$ is obtained 
by relaxing some of the conditions in 
(\ref{ep1}), (\ref{ep2}), and (\ref{ep3})
as well as the integrality of $d_-$.
The values of $n$, $d$, and $d_+$ are considered fixed parameters for $(P)$.

\begin{lemma}\label{lemma1}
In the above setting,
$${\rm OPT}(P)=
\begin{cases}
0 & \mbox{, if $d_+\leq \sqrt{dn}$}\\
d_+-\frac{(d_+-d)n}{n-d_++\sqrt{d_+^2-dn}} & \mbox{, if $d_+>\sqrt{dn}$.}
\end{cases}
$$
\end{lemma}
\begin{proof} 
Note that ${\rm OPT}(P)\geq 0$ by (\ref{ep6}).

If $d_+\leq \sqrt{dn}$, then 
$\left(d_-,\overline{d}_-,\overline{d}_+,x\right)=\left(0,0,\sqrt{dn},\sqrt{\frac{d}{n}}\right)$
is an optimal solution of $(P)$.

Now, let $d_+>\sqrt{dn}$.

Since $\left(d_-,\overline{d}_-,\overline{d}_+,x\right)$ equal to
\begin{align}\label{eopt}
\left(
d_+-\frac{(d_+-d)n}{n-d_++\sqrt{d_+^2-dn}},
d_+-\frac{(d_+-d)n}{n-d_++\sqrt{d_+^2-dn}},
d_+,
\frac{1}{n}\left(d_+-\sqrt{d_+^2-dn}\right)\right)
\end{align}
is a feasible solution of $(P)$,
we have 
$${\rm OPT}(P)\leq d_+-\frac{(d_+-d)n}{n-d_++\sqrt{d_+^2-dn}}<d.$$
In fact, it will turn out that (\ref{eopt}) is an optimal solution.

Let 
$\left(d_-,\overline{d}_-,\overline{d}_+,x\right)$ 
be an optimal solution of $(P)$.
In particular, $d_-<d$.
In view of the objective function, (\ref{ep6}) implies $d_-=\overline{d}_-$.
Since $\overline{d}_-<d<\overline{d}_+$, (\ref{ep4}) and (\ref{ep8}) imply $0<x<1$.
If $d_-=0$, then (\ref{ep5}) implies $\overline{d}_+\leq xn$, 
and, using (\ref{ep4}) and (\ref{ep7}), 
this implies the contradiction $d_+\leq \overline{d}_+\leq \sqrt{dn}$.
Hence, $d_->0$.
If $(\ref{ep5})$ would not be satisfied with equality, then 
decreasing $d_-$ and $\overline{d}_-$ both by some sufficiently small $\epsilon>0$,
and increasing $x$ by 
$\delta=\frac{(1-x)\epsilon}{\overline{d}_++\epsilon-\overline{d}_-}
\leq \frac{\epsilon}{d_+-d}$
leads to a better feasible solution, which is a contradiction.
Hence, (\ref{ep5}) is satisfied with equality.
Together with (\ref{ep4}), this implies
$d-x\overline{d}_+=(\overline{d}_+-xn)x$.
Solving this for $x$ yields
$$x\in \left\{ 
\frac{1}{n}\left(\overline{d}_++\sqrt{\overline{d}_+^2-dn}\right),
\frac{1}{n}\left(\overline{d}_+-\sqrt{\overline{d}_+^2-dn}\right)\right\}.$$
By (\ref{ep5}) satisfied with equality, we have $\overline{d}_+\geq xn$,
and, hence, $x=\frac{1}{n}\left(\overline{d}_+-\sqrt{\overline{d}_+^2-dn}\right)$.

Now, substituting this value for $x$, 
(\ref{ep4}) can be solved for $d_-$, which yields
\begin{align*}
d_- & =\overline{d}_+-\frac{\big(\overline{d}_+-d\big)n}{n-\overline{d}_++\sqrt{\overline{d}_+^2-dn}}\\
&=n\left(z-\frac{z-z_0}{1-z+\sqrt{z^2-z_0}}\right)
\mbox{ for $z=\frac{\overline{d}_+}{n}$ and $z_0=\frac{d}{n}$.}
\end{align*}
Since $d_+>\sqrt{dn}$, we have $z\geq \frac{d_+}{n}>\sqrt{z_0}$.

For the function,
$$f:\big(\sqrt{z_0},1\big]\to\mathbb{R}:z\mapsto z-\frac{z-z_0}{1-z+\sqrt{z^2-z_0}},$$
we obtain
$$f'(z)=\frac{(1-z)\big(2z^2-z_0-2z\sqrt{z^2-z_0}\big)}{\sqrt{z^2-z_0}\big(1-z+\sqrt{z^2-z_0}\big)^2}.$$
Since $f'(z)\geq 0$ for $z\in \big(\sqrt{z_0},1\big]$,
the smallest possible value of $d_-=nf(z)$ 
with $z\in \left[\frac{d_+}{n},1\right]$
is achieved for $z=\frac{d_+}{n}$,
which implies
$$d_-=d_+-\frac{(d_+-d)n}{n-d_++\sqrt{d_+^2-dn}}.$$
Altogether, it follows that 
$\left(d_-,\overline{d}_-,\overline{d}_+,x\right)$
is exactly as in (\ref{eopt}),
which completes the proof.
\end{proof}
Now, (\ref{ep9}) and Lemma \ref{lemma1} imply Theorem \ref{theorem2}.

\end{document}